\documentclass[dvips, preprint]{imsart}

\RequirePackage[OT1]{fontenc}
\RequirePackage{amsthm,amsmath, amssymb, enumerate}
\RequirePackage[numbers]{natbib}
\RequirePackage[colorlinks,citecolor=blue,urlcolor=blue]{hyperref}
\RequirePackage{hypernat}

\arxiv{math.PR/0000000}

\startlocaldefs
\numberwithin{equation}{section}
\theoremstyle{plain}
\newtheorem{theorem}{Theorem}[section]
\newtheorem{proposition}{Proposition}[section]

\newtheorem{example}{Example}[section]

\newtheorem{remark}{Remark}[section]

\endlocaldefs

\newcommand{\R}{\mathbb R}
\newcommand{\N}{\mathbb N}
\newcommand{\Z}{\mathbb Z}

\begin{document}

\begin{frontmatter}
\title{Discrete Stein characterizations and discrete information distances}%\protect\thanksref{T1}}
\runtitle{Stein characterizations and information distances}
\begin{aug}
\author{\fnms{Christophe} \snm{Ley}\thanksref[1]{1}\ead[label=e1]{chrisley@ulb.ac.be}}
\and
\author{\fnms{Yvik} \snm{Swan}\thanksref[2]{2}\ead[label=e2]{yvswan@ulb.ac.be}}
\thankstext[1]{1}{Supported by a Mandat de Charg\'e de Recherche from the Fonds National de la Recherche Scientifique, Communaut\'e fran\c{c}aise de Belgique. Christophe Ley is also member of ECARES.}
\thankstext[2]{2}{Supported by a Mandat de Charg\'e de Recherche from the Fonds National de la Recherche Scientifique, Communaut\'e fran\c{c}aise de Belgique.}
\runauthor{C. Ley and Y. Swan.}

\affiliation{ E.C.A.R.E.S., Universit\'e Libre de Bruxelles}
\address{Department of Mathematics\\
Universit\'e Libre de Bruxelles\\
Campus Plaine -- CP210\\
B-1050 Brussels\\
\printead{e1}, \printead*{e2}}
\end{aug}

\begin{abstract}
We construct two different Stein characterizations of discrete distributions and use these to provide  a natural connection between Stein characterizations for  discrete distributions and discrete  information functionals.  
\end{abstract}



\end{frontmatter}

\section{Foreword and notations}

 The purpose of this work is to  construct an explicit connection between discrete Stein characterizations and discrete information functionals (see \cite{LS11b} where similar considerations are discussed for continuous distributions).   In doing  so we also provide two general Stein characterizations of discrete distributions, as well as a family of identities relating differences between expectations with what we call \emph{generalized score functions}.    In the context of Poisson approximation, our results   allow in particular  to construct bounds between the total variation distance and (i) the so-called  \emph{scaled Fisher information} used, e.g., in \cite{KHJ05}, as well as (ii) the   \emph{discrete Fisher information} used, e.g., in \cite{JM87}.  We refer the reader to \cite{J04, MJK07} and  \cite{BJKM10} for relevant references and similar inequalities.

Throughout the paper, we shall abuse of language and call discrete probability mass functions \emph{densities}. Also, to avoid ambiguities related to division by 0, we adopt the convention that, whenever an expression involves the division by an indicator function $\mathbb{I}_A$ for some measurable set $A$, we are multiplying the expression by the said indicator function. In particular note how ratios of the form $p(x)/p(x)$ with $p(x)$ some function do not necessarily simplify to 1. Finally, we adopt the convention that  sums running over  empty sets equal 0.

%\section*{Notations and conventions}

\section{First connection} \label{sec:results1}

We start with a discrete version of the so-called \emph{density approach} (see \cite{SDHR04,  CGS10, LS11b} for a description in the continuous case). %This allows to construct Stein characterizations for   discrete distributions. 
 
 \begin{theorem}[Discrete density approach]  \label{theo:dda} Let $p$ be a density with support $S_p\subset \Z$. For the sake of convenience, we choose $S_p= [a, b] := \{a, a+1, \ldots, b\}$ with $a< b \in \Z\cup\{\pm \infty\}$. Let $\mathcal{F}_1(p)$ be the collection of all test functions $f : \Z \to \R$ such that $x\mapsto f(x)p(x)$ is bounded on $S_p$ and $f(a)=0$. Let $\Delta_x^+h(x) := h(x+1)-h(x)$ be the \emph{forward difference operator} and define  $\mathcal T_1(\cdot, p) : \Z^\star \to \R^\star$ through 
\begin{align} \mathcal T_1(f,p) : \Z \to \R : x \mapsto   \mathcal T_1(f, p)(x) := \dfrac{\Delta_x^+ (f(x) p(x))}{p(x)}\mathbb{I}_{S_p}(x). \label{eq:stein_dda_op}\end{align} 
Let $Z \sim p$ and let $X$ be a real-valued discrete random variable. 
\begin{enumerate}[(1)]
\item If $X \stackrel{\mathcal L}{=}ÊZ$ then ${\rm E}\left[\mathcal{T}_1(f, p)(X) \right]=0$ for all $f \in \mathcal{F}_1(p)$.
\item  If   ${\rm E}[\mathcal{T}_1(f, p)(X)]=0$ for all $f\in\mathcal{F}_1(p)$, then $X\,|\,X\in S_p\stackrel{\mathcal{L}}{=} Z$.
 \end{enumerate}
%
%Then
%\begin{align} q= p  \Longleftrightarrow {\rm E}_q\left[\mathcal T_1(f,p)(X)\right]  = 0 \mbox{ for all } f \in \mathcal{F}(p) \label{eq:stein_dda_ca}\end{align} 
%where, here and throughout, we write  ${\rm E}_q[h(X)] = \sum_{j \in S_q}h(j) q(j)$.  
\end{theorem}

We draw the reader's attention to the similarity between the operator $\mathcal T_1$ and the  operators introduced in \cite{LS11a, LS11b}:  in the terminology of \cite{LS11a},  our operator \eqref{eq:stein_dda_op} allows for a discrete ``location''-based parametric interpretation. % See also \cite{LPU02} for a similar operator. %Technicalities induced   problems with the support encouraged us not to include these results in \cite{LS11a}. 

%[Proof of Theorem \ref{theo:dda}]

\begin{proof} The first statement is trivial. To see (2), consider for $z\in \Z$ the functions $f_z^p$ defined through
$$ f_z^p: \Z\to \R: x \mapsto \frac{1}{p(x)} \sum_{k=a}^{x-1} l_z(k)p(k)$$
with $l_z(k):= ({\mathbb{I}}_{(- \infty, z]}(k) - {\rm P}_p(X \le z))\mathbb{I}_{S_p}(k)$ and  ${\rm P}_p(X\le z):=\sum_{k=-\infty}^z p(k)$.  
It is evident that $x\mapsto f_z^p(x)p(x)$ is bounded and that $f_z^p(a)=0$ by our convention on sums, hence $f_z^p \in \mathcal{F}_1(p)$ for all $z$. Moreover we have $\Delta_x^+(f_z^p(x)p(x))= l_z(x)p(x)$. This result is direct for $x<b$; for $x=b$, $\Delta_x^+(f_z^p(x)p(x))|_{x=b}=f_z^p(b+1)p(b+1)-f_z^p(b)p(b)=- \sum_{k=a}^{b-1} l_z(k)p(k)=l_z(b)p(b)$ since $\sum_{k=a}^{b} l_z(k)p(k)=0$ by definition of $l_z$. It follows  that this forward difference satisfies, for all $z$,  the so-called \emph{Stein equation}
$$
\mathcal{T}_1(f_z^p, p)(x) = l_z(x).
$$
Consequently, we can use ${\rm E}\left[\mathcal{T}_1(f_z^p, p)(X)\right] =0$ to obtain   
$${\rm P}(X\leq z\cap X\in S_p)={\rm P}(X\in S_p){\rm P}_p(X\leq z)$$
 for all $z\in\Z$. In other words, provided that ${\rm P}(X\in S_p)>0$, ${\rm P}(X\leq z\,|\, X\in S_p)={\rm P}_p(X\leq z)={\rm P}(Z\leq z)$ for all $z\in\Z$, whence the claim.
 \end{proof}

Note that the choice of a ``connected'' support is for convenience only, and straightforward arguments allow to adapt the result to supports of the form $[a, b] \cup [c, d]$ with $c>b$. Likewise the use of a forward difference  in the expression of the operator is purely arbitrary and minor adaptations (e.g., setting $f(b)=0$ instead of $f(a)=0$) allow to reformulate \eqref{eq:stein_dda_op}   in terms of backward differences as well.

\begin{example} It is perhaps informative to see how   the operator $\mathcal{T}_1(f,p)$ spells out in certain specific examples. 
\begin{enumerate}[1.]
\item Take $p(x)= e^{-\lambda} \lambda^x/x!\,\mathbb{I}_\N(x)$ the   density of a mean-$\lambda$ Poisson random variable. Then (abusing notations) $\mathcal{F}_1(Po(\lambda))$ contains the set of bounded functions $f$ with $f(0)=0$, and simple computations show that the operator becomes 
$$\mathcal T_1(f, Po(\lambda))(x) = \left(\frac{\lambda}{x+1}f(x+1) - f(x) \right)\mathbb{I}_\N(x).$$
\item Take $p$ to be a member of Ord's family, i.e. suppose that there exist $s(x)$ and $\tau(x)$ such that 
$$\frac{p(x+1)}{p(x)} = \frac{s(x) + \tau(x)}{s(x+1)}.$$
For an explanation on these notations see \cite{S01}. The collection $\mathcal{F}_1((s,\tau))$ contains the set of all functions of the form $f(x)=f_0(x)s(x)$ with $f_0$ bounded and, for these $f$, the operator writes out 
$$\mathcal T_1(f, (s, \tau))(x) = (s(x)+\tau(x))f_0(x+1)\mathbb{I}_{[a,b]}(x+1) - f_0(x)s(x)\mathbb{I}_{[a,b]}(x).$$
We retrieve, up to some minor modifications, the operator presented in \cite{S01}; using the backward difference operator and functions $f$ of the form $f_0(x)(s(x)+\tau(x))$ yields exactly the operator proposed in that paper.
\item Write $p$ as a Gibbs measure, i.e.  $p(x) = e^{V(x)} \omega^x/(x! \mathcal Z)\mathbb{I}_{[0,N]}(x)$ with $N$ some positive integer, $\omega>0$ fixed, $V$ a function mapping $\N$ to $\R$ and $\mathcal Z$ the normalizing constant. For an explanation on the notations see \cite{ER08}. The collection $\mathcal{F}_1((V,\omega))$ contains the set of all functions of the form $f(x)=xf_0(x)$ with $f_0$ bounded and, for these $f$, the operator is of the form $\mathcal T_1(f, (V, \omega))(x) = f_0(x+1)e^{V(x+1)-V(x)} \omega \mathbb{I}_{[0,N]}(x+1)- xf_0(x)\mathbb{I}_{[0, N]}(x)$. Supposing, as in \cite{ER08}, that $f_0(N+1)=0$, the latter operator simplifies to
$$\mathcal T_1(f, (V, \omega))(x) = \left(f_0(x+1)e^{V(x+1)-V(x)} \omega - xf_0(x)\right)\mathbb{I}_{[0, N]}(x),$$
which corresponds   to the Stein operator presented in \cite{ER08}. 
\end{enumerate}
\end{example}

Following the methodology introduced in \cite{LS11b}, the next step consists in uncovering  a  factorization property  of the operator~\eqref{eq:stein_dda_op} for two densities $p$ and~$q$. 
It will be fruitful  to consider distributions $p$ and $q$ having   non-equal supports. We choose to fix, for the sake of convenience (and for this sake only),  $S_p = \N$ and $S_q=[0,\ldots,N]$, for $N \in \N \cup \{\infty\}$. Then, since we can   always write  $1 =   {q(x)}/{q(x)} + \mathbb I_{[N+1, \ldots]}(x)$,  we get
  \begin{align}\label{eq:step1}  \mathcal T_1(f,p)(x) & =  \frac{\Delta^+_x \left(f(x) p(x) \dfrac{q(x)}{q(x)}\right)}{p(x)} + \frac{\Delta_x^+\left(f(x) p(x) \mathbb I_{[N+1, \ldots]}(x)\right)}{p(x)}. \end{align}
%The second term on the rhs of \eqref{eq:step1} yields the error term $e_{f, p}^N$ in \eqref{eq:fac_prop_1}. As for the first term on the rhs of \eqref{eq:step1}, 
 Now recall  the  product rule for discrete derivatives  $$\Delta_x^+ (h(x) g(x)) = h(x+1) \Delta_x^+ g(x)+g(x) \Delta_x^+ h(x).$$
 Applying this and keeping in mind that we have set $S_q\subset S_p$, the    first term on the rhs of \eqref{eq:step1} becomes
  \begin{align*} &    \frac{f(x+1)q(x+1)}{p(x)} \Delta_x^+ \left(\dfrac{p(x)}{q(x)}\right)  + \frac{\Delta_x^+(f(x)q(x))}{p(x)}  \dfrac{p(x)}{q(x)} \\
														& \quad \quad =  f(x+1)  \left(\dfrac{p(x+1)}{p(x)}- \dfrac{q(x+1)}{q(x)}\right)\mathbb I_{[0,\ldots,N]}(x)  + \frac{\Delta_x^+(f(x)q(x))}{q(x)}. \end{align*}
Therefore, letting 
% Then the following holds. 
%\begin{proposition}%[Factorization Theorem for Discrete Stein Operators] 
% \label{theo:facto_1} Let $p$ and $q$ be two discrete densities  as above. Define the standardized score function 
\begin{equation} \label{eq:score_d1} r_1(p,q)(x) := \left(\frac{p(x+1)}{p(x)} - \frac{q(x+1)}{q(x)}\right)\mathbb I_{[0,\ldots,N]}(x),\end{equation}
we have just shown that, for all   $f \in \mathcal{F}_1(p)\cap \mathcal{F}_1(q)$, we have the \emph{factorization property}
\begin{equation}\label{eq:fac_prop_1} \mathcal T_1(f,p)(x) = \mathcal T_1(f,q)(x) + f(x+1) r_1(p,q)(x) +e^N_{f,p}(x),\end{equation}
where $$e^N_{f, p}(x) := f(x+1)p(x+1)/p(x) \mathbb I_{[N, \ldots]}(x) - f(x) \mathbb I_{[N+1, \ldots]}(x).$$

\begin{remark} The statements above (and their consequences) are easily adapted to situations where $S_p \subset S_q$; having in mind the context of a Poisson target $p$ explains our willingness to restrict  our choice. \end{remark}
%[Proof of Theorem \ref{theo:facto_1}]
%
% \begin{proof}   
%%\frac{\Delta^+_x \left(f(x) q(x) \dfrac{p(x)}{q(x)}\right)}{p(x)} & 
%The claim directly follows. 
%\end{proof}
Now  let $l:\Z\to\R$ be a function such that ${\rm E}_p[l(X)]$ and ${\rm E}_q[l(X)]$ exist, with ${\rm E}_r[l(X)]:=\sum_{k\in S_r}l(k)r(k)$ for a density $r$ with support $S_r$. Still following \cite{LS11b}, it is immediate  that the function 
\begin{equation}\label{eq:steineq_sol3} 
f_{1, l}^p : \Z \to \R : x \mapsto \frac{1}{p(x)}\sum_{k=0}^{x-1}(l(k)- {\rm E}_p[l(X)]) p(k)\end{equation}
is solution of the so-called Stein equation $\mathcal T_1(f, p)(x) = l(x) - {\rm E}_p[l(X)]$, so that, taking expectations and using~(\ref{eq:fac_prop_1}), we get 
%This yields our next  result, whose  proof is immediate and hence left to the reader. 
%\begin{proposition}\label{prop:steinentropy2}   
%et $p$ and $q$ be as above and
% letting $l:\Z\to\R$ be a function such that ${\rm E}_p[l(X)]$ and ${\rm E}_q[l(X)]$ exist. Define 
%\begin{equation}\label{eq:steineq_sol3} 
%f_{1, l}^p : \Z \to \R : x \mapsto \frac{1}{p(x)}\sum_{k=0}^{x-1}(l(k)- {\rm E}_p[l(X)]) p(k).\end{equation}
%Then 
\begin{equation} \label{eq:fundeq2}  {\rm E}_q[l(X)] -{\rm E}_p[l(X)] =  
{\rm E}_q [{f}_{1,l}^p(X+1) r_1(p,q)(X)] +e^N_{p,q}(l),
\end{equation}
with $e^N_{p,q}(l) :=  q(N)f_{1, l}^p(N+1)p(N+1)/p(N).$
% \end{proposition}
 
\begin{remark} The error term $e^N_{p,q}(l)$ in \eqref{eq:fundeq2} will be negligible as $N$ tends to infinity since, in general, the Stein solution $f_{1, l}^p$ will be bounded over $\N$. This latter fact also ensures that $f_{1, l}^p$ belongs to $\mathcal{F}_1(q)$. \end{remark}

We will apply \eqref{eq:fundeq2} in the context of a Poisson target distribution in Section~\ref{sec:appli}.  In particular we will show how our approach provides a connection between the so-called \emph{total variation distance} (as well as many other probability distances)  and the  scaled Fisher information in use for information theoretic approaches to Poisson approximation problems (see \cite{KHJ05, MJK07, BJKM10}).

\section{A second connection} \label{sec:results2}

The construction from the previous  section  (i.e. the  factorization \eqref{eq:fac_prop_1}, the score function \eqref{eq:score_d1} and the identity \eqref{eq:fundeq2}) is by no means unique, nor is the initial characterization from Theorem \ref{theo:dda}. There are, in fact,  an infinite number of variations on the different steps outlined above, each providing a connection between probability distances and different forms of  information distances. Now  it appears that, in the world of Poisson approximation,  the  scaled Fisher information is not the only ``natural'' measure of discrepancy  and  \cite{JM87} (followed later by \cite{BJKM10}) make use of another information  distance which they call  the discrete Fisher information.  We choose to show   how this specific distance can be obtained from our Stein characterizations as well. 

  In \cite{LS11a} we propose a construction of  Stein characterizations tailored for \emph{parametric   densities}, that is   densities depending on some real-valued parameter. 
In what follows, we shall denote by $p_\theta(x)$ the parametric density with parameter $\theta$ belonging to the parameter space $\Theta$. For the sake of simplicity we consider families with support $S_p= [a, b] := \{a, a+1, \ldots, b\}$ with $a< b \in \Z\cup\{\pm \infty\}$  not depending on $\theta$; we also suppose that, for all $x$, the  function $\theta\mapsto p_\theta(x)$ is continuously differentiable. (A similar result can also be obtained for integer-valued parameters $\theta$.) We then obtain the following result, whose proof is omitted because it is directly inspired from~\cite{LS11a} and  runs  along the same lines as the proof of Theorem~\ref{theo:dda}.

\begin{theorem}[Parametric discrete density approach] \label{theo:pdda} For $\theta$ an interior point of $\Theta$, let $p(x):=p_\theta(x)$ be a parametric density with support $S_p\subset\Z$ and define $\tilde{p}(x) := \partial_\theta (p_{\theta}(x) /p_{\theta}(a) )$.      Let $\mathcal{F}_2(p)$ be the collection of all test functions $f:\Z\rightarrow\R$ such that $x\mapsto f(x)\tilde{p}(x)$ is bounded on $S_p$. Define the operator $\mathcal{T}_2(\cdot,p):\Z^*\rightarrow\R^*$ through 
$$
\mathcal{T}_2(f,p):\Z\rightarrow\R:x\mapsto \mathcal{T}_2(f,p)(x):=\dfrac{ \Delta_x^+(f(x)\tilde{p}(x))}{p(x)}\mathbb{I}_{S_p}(x).
$$
Let $Z \sim p$ and let $X$ be a real-valued discrete random variable. 
\begin{enumerate}[(1)]
\item If $X \stackrel{\mathcal L}{=}ÊZ$ then ${\rm E}\left[\mathcal{T}_2(f, p)(X) \right]=0$ for all $f \in \mathcal{F}_2(p)$.
\item  If   ${\rm E}[\mathcal{T}_2(f, p)(X)]=0$ for all $f\in\mathcal{F}_2(p)$, then $X\,|\,X\in S_p\stackrel{\mathcal{L}}{=} Z$.
 \end{enumerate}
\end{theorem}

We attract the reader's attention to the fact that, contrarily to $\mathcal{F}_1(p)$ in Theorem~\ref{theo:dda}, the class of test functions $\mathcal{F}_2(p)$ here does not ask that $f(a)=0$. This comes from the fact that, by definition, $\tilde{p}(a)=0$, hence this requirement on the $f$ can be dropped.

Theorem~\ref{theo:pdda} allows to recover the well-known Stein operators and characterizations of the Poisson, geometric, binomial distributions, to cite but these; we refer the reader to \cite{LS11a} for intuition about the perhaps unusual form of the operator, as well as for explicit computations and examples. 

From here onwards we restrict our attention to distributions $p$ and $q$ with full support $\N$. Note that this entails that $\tilde{p}(x)$ and $q(x-1)$ share the same support $\N_0$. While not strictly necessary, this  assumption will yield considerable simplifications. It is, moreover,  in line with the related literature when a Poisson target is to be considered (see \cite{BJKM10}). 

Proceeding as in Section~\ref{sec:results1} (and keeping all supports implicit) we readily obtain 
\begin{align*}Ê  \mathcal T_2(f,p)(x) & =  \frac{ \Delta_x^+\left(f(x)q(x-1)\frac{\tilde{p}(x)}{q(x-1)}\right)}{p(x)} \\
					    & = \frac{ \Delta_x^+(f(x)q(x-1))}{q(x)}\frac{\tilde{p}(x+1)}{p(x)} + f(x) \frac{q(x-1)}{p(x)}   { \Delta_x^+\left(\frac{\tilde{p}(x)}{q(x-1)}\right)}. \end{align*}		
Straightforward simplifications then yield for $f\in\mathcal{F}_2(p)\cap\mathcal{F}_2(q)$ the factorization
\begin{equation}\label{eq:fac_prop_2} \mathcal T_2(f,p)(x) =  f(x) r_2(p,q)(x) + \frac{\Delta^+_x \left( f(x) q(x-1) \right)}{q(x)}\frac{{\tilde{p}(x+1)}}{p(x)}\end{equation}
with 
\begin{equation} \label{eq:score_d2} r_2(p, q)(x) := \frac{\tilde p(x+1)}{p(x)}  \frac{q(x-1)}{q(x)} - \frac{\tilde{p}(x)}{p(x)}.\end{equation}
  
  Now let $l:\Z\to \R$ be a function such that ${\rm E}_{p} [l(X)]$ and ${\rm E}_q [l(X)]$ exist, and define
\begin{equation}\label{eq:steineq_sol4} 
f_{2,l}^p(x) :\Z\rightarrow\R:x\mapsto \frac{1}{\tilde{p}(x)}\sum_{k=0}^{x-1}(l(k)- {\rm E}_p [l(X)]) p(k).\end{equation}
Then clearly $\mathcal T_2(f_{2,l}^p, p)(x) =l(x)- {\rm E}_p [l(X)]$ so that,  taking expectations on both sides of \eqref{eq:fac_prop_2} for this choice of test function, we obtain
\begin{align*}{\rm E}_q[l(X)] -{\rm E}_{p}[l(X)] & = 
{\rm E}_q [{f}_{2,l}^p(X) r_2(p, q)(X)]  \\
& \quad \quad \quad \quad +{\rm E}_q \left[ \frac{\Delta^+_x \left.\left( f_{2,l}^p(x) q(x-1) \right)\right|_{x=X}}{q(X)}\frac{{\tilde{p}(X+1)}}{p(X)}\right].
\end{align*}
Finally suppose that $ {{\tilde{p}(X+1)}}/{p(X)}$ simplifies to a constant (as is the case for a Poisson target). Then straightforward calculations lead to the analog of \eqref{eq:fundeq2} for the  score function $r_2$, namely 
\begin{equation} \label{eq:fundeq4}  {\rm E}_q[l(X)] -{\rm E}_{p}[l(X)] = 
{\rm E}_q [{f}_{2,l}^p(X) r_2(p, q)(X)].
\end{equation}
% we focus on this particular choice is that, a
As will be shown in  Section~\ref{sec:appli}, specifying a Poisson distribution for the target $p$ in \eqref{eq:fundeq4} yields the scaled score function whose variance is the so-called discrete Fisher information introduced in \cite{JM87}. 

\section{Applications to a Poisson target} \label{sec:appli}

Working as in \cite{LS11b}  it is easy to obtain,  from \eqref{eq:fundeq2} and \eqref{eq:fundeq4},   inequalities of the form
$$d_\mathcal{H}(p, q) :=  \sup_{h \in \mathcal H} \left| {\rm E}_q[h(X)] - {\rm E}_{p}[h(X)]\right| \le \kappa(p, q) \mathcal J(p,q),$$
where $\mathcal{H}$ is, as usual, a suitably chosen class of functions, $\kappa(p, q)$ are constants depending on both $p$ and $q$ and $\mathcal J(p,q)$ is a so-called \emph{information distance} between $p$ and $q$, which is given by the variance of one of the score functions \eqref{eq:score_d1} or \eqref{eq:score_d2} introduced in the two previous sections. The main difficulty then resides in computing the constants appearing in these inequalities and in putting the information distance $\mathcal{J}$ to good use.  Such computations are not the primary purpose of the present paper. Hence we  choose to focus on a Poisson target,  for which much is already known. From here onwards we therefore only consider $p = Po(\lambda)$, the mean-$\lambda$ Poisson density.

We first adapt the results from Section \ref{sec:results1}. The score function \eqref{eq:score_d1} becomes 
$$r_1(Po(\lambda),q)(x) = \frac{\lambda}{x+1} \left(1- \frac{(x+1)q(x+1)}{\lambda q(x)}\right)\mathbb{I}_{[0,\ldots,N]}(x)$$
so that  \eqref{eq:fundeq2} yields
\begin{equation} \label{eq:poissoncase1}   {\rm E}_q[l(X)] - {\rm E}_p[l(X)] =  {\rm E}_q \left[ \left(\frac{\sqrt{\lambda}{f}_{1,l}^p(X+1)}{X+1}\right) \sqrt{\lambda}\left(1- \frac{(X+1)q(X+1)}{\lambda q(X)} \right)\right]+e^N_{p,q}(l).
\end{equation}
One recognizes, in the rhs of \eqref{eq:poissoncase1},  the scaled score function whose variance yields the  scaled Fisher information
$$\mathcal K_1(Po(\lambda),q) := \lambda {\rm E}_q\left[\left(\frac{(X+1)q(X+1)}{\lambda q(X)} - 1\right)^2\right].$$ 
This information distance is subadditive over convolutions; this is  useful  when computing rates of convergence for sums  towards the Poisson distribution  (see, e.g., \cite{KHJ05, BJKM10}).  Using a  Poincar\'e inequality, \cite{KHJ05} show that, for $q$ a discrete distribution with mean $\lambda$,   
$$\| q- Po(\lambda)\|_{TV} \le \sqrt{2 \mathcal K_1(Po(\lambda),q)},$$
with $\|\cdot\|_{TV}$ indicating the \emph{total variation distance}. From \eqref{eq:poissoncase1} and H\"{o}lder's inequality we obviously recover a much more general result, namely 
\begin{align}\label{eq:holder1} d_{\mathcal{H}}(Po(\lambda), q) & = \sup_{h \in \mathcal H} \left| {\rm E}_q[h(X)] - {\rm E}_{Po(\lambda)}[h(X)]\right|\nonumber \\
												  & \le H_{1,\mathcal H}( Po(\lambda), q)\sqrt{\mathcal K_1(Po(\lambda),q)},\end{align}
where the constant 
$$H_{1,\mathcal H}(Po(\lambda), q) :=\sup_{h\in\mathcal{H}}\left( \sqrt{{\rm E}_q\left[\left(\frac{\sqrt \lambda {f}_{1,h}^p(X+1)}{X+1}\right)^2\right]}+\frac{e^N_{p,q}(h)}{\sqrt{\mathcal K_1(Po(\lambda),q)}}\right)$$
is some kind of general Stein (magic) factor. The  notation $H$ for these constants is borrowed from \cite{BJKM10}  where  similar relationships are obtained, within the context of   compound Poisson approximation. 

Likewise, in the notations of Section \ref{sec:results2},  we have
$\tilde{p}(x) =  {\lambda^{x-1}}/{(x-1)!}\mathbb{I}_{\N_0}(x)$
so that 
$ {{\tilde{p}(x+1)}}/{p(x)}  = e^{\lambda}\mathbb{I}_{\N}(x)$ and ${{\tilde{p}(x)}}/{p(x)}  =  e^\lambda{x}/{\lambda} \mathbb{I}_{\N}(x).$
Hence for all $q$ with full support $\N$ we get
%\begin{align*} {\rm E}_q \left[ \frac{\Delta^+_x \left( f_{2,l}^p(x) q(x-1) \right)}{q(x)}\frac{{\tilde{p}(x+1)}}{p(x)}\right] & = e^{\lambda}  {\rm E}_q \left[ \frac{\Delta^+_x \left( f_{2,l}^p(x) q(x-1) \right)}{q(x)} \right]  =0 \end{align*} and 
$$r_2(Po(\lambda), q)(x) = e^{\lambda} \left(\frac{q(x-1)}{q(x)} - \frac{x}{\lambda}\right)\mathbb{I}_\N(x)$$
 so that  \eqref{eq:fundeq4} yields
\begin{equation} \label{eq:poissoncase2}   {\rm E}_q[l(X)] - {\rm E}_p[l(X)] =   {\rm E}_q \left[\left( \lambda^{-1} e^\lambda{f}_{2,l}^p(X)\right)\left(\frac{\lambda q(X-1)}{q(X)} -  {X}\right)\right].
\end{equation}
One recognizes, in the rhs of \eqref{eq:poissoncase2},  a special instance of the  Katti-Panjer score function introduced in \cite[equation (3.1)]{BJKM10} and whose variance yields our second  information distance, namely the discrete Fisher information
\begin{equation}\label{eq:kpinfo} \mathcal K_2(Po(\lambda),q) := {\rm E}_q\left[\left(\frac{\lambda q(X-1)}{q(X)} -  {X}\right)^2\right].\end{equation}
This is easily  shown  to be related to  the  \emph{discrete Fisher information distance} $I(q) := {\rm E}_q [ ( {q(X-1)}/{q(X)}-1 )^2 ]$  introduced in  \cite{JM87}. 
The information distance \eqref{eq:kpinfo} has been shown to be subadditive over convolutions (see \cite{BJKM10}). 
From \eqref{eq:poissoncase2} and H\"{o}lder's identity we obviously recover the following general relationship
\begin{align}\label{eq:holder2} d_{\mathcal{H}}(Po(\lambda), q) & = \sup_{h \in \mathcal H} \left| {\rm E}_q[h(X)] - {\rm E}_{Po(\lambda)}[h(X)]\right| \nonumber \\
												   & 	\le H_{2, \mathcal H}(Po(\lambda), q)\sqrt{\mathcal K_2(Po(\lambda),q)},\end{align}
where the constant
$$H_{2,\mathcal H}(Po(\lambda), q) := \sqrt{\sup_{h \in \mathcal H} {\rm E}_q\left[\left(\lambda^{-1} e^{\lambda} {{f}_{2,h}^p(X)}\right)^2\right]}$$
is, again,  some kind of general Stein (magic) factor.

We conclude the paper with   explicit computations.

 \begin{proposition}\label{cor:bounds1} Take $p=Po(\lambda)$  and $q$ a pdf with support $[0,\ldots,N]$. Then 
\begin{equation}\label{eq:bound3}\| p-q\|_{TV} \le \sqrt{\lambda} H(\lambda) \sqrt{\mathcal K_1(Po(\lambda),q)}+e^N_{q} \text{ and } \| p-q\|_{TV} \le H(\lambda) \sqrt{\mathcal K_2(Po(\lambda),q)}, \end{equation} 
 where %$e_N \propto q(N)$ and 
the error term $e^N_q$ is of order $q(N)/(N+1)$ and $H(\lambda)  =  1 \wedge \sqrt{\frac{2}{e {\lambda}}}$. The second bound in \eqref{eq:bound3}  only holds if  $N=\infty$. 
 \end{proposition}

% [Proof of Corollary \ref{cor:bounds1}] 
 \begin{proof}
Choose  
$$h(x) := \mathbb{I}_{[p(x)\le q(x)]} -  \mathbb{I}_{[p(x)\ge q(x)]} = 2 \mathbb{I}_{[p(x)\le q(x)]} - 1.$$  Then obviously ${\rm E}_p[h(X)]$ and ${\rm E}_q[h(X)]$ exist, and 
$$ \sum_x | p(x) - q(x) |  =   {\rm E}_q[ h(X)] - {\rm E}_p [h(X)]$$   so that, by definition of the total variation distance, we get 
$$ \| p-q\|_{TV}  =   \frac{1}{2} \left({\rm E}_q \left[h(X)\right] - {\rm E}_p  \left[h(X)\right]\right).$$
It now suffices to apply \eqref{eq:holder1} and \eqref{eq:holder2}, respectively, to obtain   the announced relationships. All that remains is to compute bounds on the constants. 

In the first case, known results on the properties of ${f}_{1,h}^p$ show that the claim on the error term is evident.  The expression for the constant $\sqrt{\lambda} H(\lambda)$ is derived from the quantity
$${\rm E}_q\left[\left(\frac{\sqrt \lambda {f}_{1,h}^p(X+1)}{X+1}\right)^2\right], $$
with $h$ specified (and bounded by 2). Indeed, from  \eqref{eq:steineq_sol3} and   \cite[Theorem 2.3]{E05}, we get 
\begin{align*}\left\|\frac{{f}_{1,h}^p(x+1)}{x+1}\right\|  & = \left\| \frac{x!}{\lambda^{x+1}} \sum_{k=0}^x(h(k) -  {\rm E}_p[h(X)]) \frac{\lambda^k}{k!}\right\| \\
										& \le \left(1 \wedge  \sqrt{\frac{2}{e\lambda}}\right) \left( \sup_{i \in \N} h(i) - \inf_{i \in \N} h(i)\right).\end{align*}
The constant $H(\lambda)$ in the second case is derived from 
$${\rm E}_q\left[\left(\lambda^{-1} e^{\lambda} {{f}_{2,h}^p(X+1)}\right)^2\right].$$
Actually, from  (\ref{eq:steineq_sol4})  and \cite[Theorem 2.3]{E05} we get 
\begin{align*}\left\|  \lambda^{-1} e^\lambda{{f}_{2,h}^p(x)}\right\|  & = \left\| \frac{(x-1)!}{\lambda^{x}} \sum_{k=0}^{x-1}(h(k) -  {\rm E}_p[h(X)]) \frac{\lambda^k}{k!}\right\| \\
										& \le \left(1 \wedge  \sqrt{\frac{2}{e\lambda}}\right) \left( \sup_{i \in \N} h(i) - \inf_{i \in \N} h(i)\right).\end{align*}
The claim follows. 
 \end{proof}

For $\lambda<2/e$, $H(\lambda)=1$ and hence the bounding constant for the scaled Fisher information $\mathcal K_1(Po(\lambda),q)$  becomes $\sqrt{\lambda}<\sqrt{2/e}$; in case $\lambda>2/e$, this constant equals $\sqrt{2/e}$. Since the error term $e^N_q$ is either null for $N=\infty$ or negligible in comparison to the term involving the scaled Fisher information, our bounds on the total variation distance corresponding to the first inequality in~(\ref{eq:bound3}) improve on those proposed in~\cite{KHJ05}, where the bounding constant is given by $\sqrt{2}$, while ours are inferior to $\sqrt{2/e}$. For the sake of illustration we conclude this section by applying Proposition~\ref{cor:bounds1} to the three examples studied in \cite{KHJ05}.

\begin{example}
Take $X_i$ i.i.d. Bernoulli~$(\lambda/n)$ random variables and let $S_n = \sum_{i=1}^n X_i$. Put $q=P_{S_n}$, the density associated with the sum $S_n$. Then straightforward calculations reveal that $\mathcal{K}_1(Po(\lambda),q)=\lambda^2/(n(n-\lambda))$ and $e^n_q$ is of order $\lambda^n/n^{n+1}$. Consequently, we have
$$\| P_{S_n} - Po(\lambda)\|_{TV} \le \sqrt{\frac{2}{e}}\frac{\lambda}{\sqrt{n(n-\lambda)}} +c\frac{\lambda^n}{n^{n+1}}$$
for some positive constant $c$ and sufficiently large $n$. This is an improvement over the $(2+\epsilon)\lambda/n,\epsilon>0,$ bound obtained in \cite{KHJ05}.
\end{example}

\begin{example}
Consider the same situation as above, but with $\lambda$ replaced by $\mu\sqrt{n}$ for some $\mu>0$. From the previous example, we directly deduce that 
$$\| P_{S_n} - Po(\mu\sqrt{n})\|_{TV} \le \sqrt{\frac{2}{e}}\frac{\mu}{\sqrt{\sqrt{n}(\sqrt{n}-\mu)}} +\frac{\mu^n}{n^{n/2+1}}$$
for some positive constant $c$ and sufficiently large $n$. Although the rate is good and the constant above  is again an improvement over the one  obtained in \cite{KHJ05}, it is still not as good as the optimal constant $\sqrt{1/(2\pi e)}$ derived in \cite{DP86}.
\end{example}

\begin{example}
Finally take $X_i$ independent geometric random variables with respective distributions $P_i(x)=(1-q_i)^xq_i\mathbb{I}_\N(x)$, where $0\leq q_i\leq1$ for all $i=1,\ldots,n$. Let $S_n = \sum_{i=1}^n X_i$ and $q=P_{S_n}$, the density associated with the sum $S_n$. Put $\lambda={\rm E}[S_n]$. The subadditivity property of $\mathcal{K}_1(Po(\lambda),q)$ states that (see \cite[Proposition 3]{KHJ05})
$$\mathcal{K}_1(Po(\lambda),P_{S_n})\leq \sum_{i=1}^n\frac{(1-q_i)}{\lambda q_i}\mathcal{K}_1(Po(e_i),P_{X_i}),$$
where $P_{X_i}$ is the density associated with $X_i$ and $e_i = {\rm E}[X_i]$. Straightforward computations show that   $\mathcal{K}_1(Po(e_i),P_{X_i})=(1-q_i)^2/q_i$. Since here $e^\infty_q=0$, it follows that
$$\| P_{S_n} - Po(\lambda)\|_{TV} \le \sqrt{\frac{2}{\lambda e}}\sqrt{\sum_{i=1}^n\frac{(1-q_i)^3}{q_i^2}}$$
for sufficiently large $n$. Again we improve on  the  constant obtained in  \cite{KHJ05}. Note that restricting, as in \cite{KHJ05},  to the case where  $q_i = n/(n+\lambda)$ yields a rate of $\sqrt{2/e} (\lambda/\sqrt{n(n+\lambda)})$.  

Next consider the second information functional $\mathcal{K}_2$. %, which also satisfies a   subadditivity property    (see \cite[Proposition 4.3]{BJKM10})  
%$$\mathcal{K}_2(Po(\lambda),P_{S_n})\leq \sum_{i=1}^n\mathcal{K}_2(Po(e_i),P_{X_i}),$$
%with the same notations as above. 
Direct computations yield an expression for  $\mathcal{K}_2(Po(\lambda),P_{S_n})$ which we will dispense of here, and hence an explicit  bound on $\| P_{S_n} - Po(\lambda)\|_{TV}$ can also easily be obtained  in terms of this functional as well. The general expression appears inscrutable, and hence we restricted our attention to the case where  $q_i = n/(n+\lambda)$. There,   numerical evaluations in Mathematica 7    encourage us to suggest that the second information distance   provides a better  rate   than the  $\sqrt{2/e}(\lambda/\sqrt{n(n+\lambda)})$ mentioned above, at least for moderate values of $\lambda$ and large  values of  $n$ (that is,  $n \ge 100)$. 
\end{example}

%In all three examples, we have only given the bound corresponding to the scaled Fisher information inequality. In the first two examples this is evident, as the inequality related to the Katti-Panjer score function requires that $q$ has full support $\N$. Regarding the third example, this second bound does not produce a better result than the one presented above, hence we have omitted it.

\section{Final comments}

The results reported  in the present work are to be read in conjunction with those reported in \cite{LS11b}. The main message of these two papers is that  all the so-called Fisher information functionals   used in the literature on Gaussian and Poisson approximation bear an interpretation in terms of  a specific Stein characterization.  As  concluding remark to the present paper we wish to stress the fact that   our method    applies to many more distributions than just the Gaussian or the Poisson (e.g., the compound Poisson, allowing comparisons with the results of~\cite{BJKM10}), and in particular provides generalized scaled Fisher information distances between any two (nice) distributions. Of course much remains to be explored, in particular on the properties of these generalized information functionals.  However  the freedom of choice for the densities as well as for the test functions in  \eqref{eq:fundeq2},  \eqref{eq:fundeq4} and \cite[Theorem 2.3]{LS11b} makes us confident that there remains much to be gained from a crafty usage of such  identities.

\end{document}